\RequirePackage{fix-cm}
\documentclass[smallextended]{svjour3}
\smartqed

\usepackage{amsmath,amssymb,mathtools}
\usepackage{booktabs}
\usepackage{graphicx}
\usepackage{url}
\usepackage[hidelinks]{hyperref}

\journalname{Geometriae Dedicata}

\newcommand{\R}{\mathbb{R}}
\newcommand{\Dnat}{\mathcal D_N}
\newcommand{\area}{\operatorname{area}}
\newcommand{\diameter}{\operatorname{diam}}
\newcommand{\conv}{\operatorname{conv}}
\newcommand{\dd}{\,\mathrm{d}}
\newcommand{\eps}{\varepsilon}
\newcommand{\cross}{\mathbin{\times}}

\begin{document}

\title{Maximum-Area Small Polygons of Even Order}

\titlerunning{Maximum-Area Small Polygons of Even Order}

\author{Dawid Trela}
\authorrunning{D. Trela}

\institute{D. Trela \at
  Faculty of Law and Administration, War Studies University, Warsaw, Poland\\
  \email{dawidmtrela@gmail.com}\\
  ORCID: \href{https://orcid.org/0000-0001-9781-6425}{0000-0001-9781-6425}}

\date{}

\maketitle

\begin{abstract}
A small $n$-gon is a planar $n$-gon whose diameter is at most one.  For odd $n$, Reinhardt proved that the regular polygon is optimal.  For even $n$, the maximizer is nonregular, and only a few low orders were known exactly.  We prove that for every even $n\geq8$ the maximum-area small $n$-gon is unique up to Euclidean isometry and reflection.  Foster and Szab\'o's description of the diameter graph reduces a maximizer to an $(n-1)$-cycle of unit distances together with one pendant diameter.  We determine the compatible boundary order, interpret the cycle as the centers of a Reuleaux $(n-1)$-gon, and eliminate the pendant vertex by a one-variable area calculation.  The remaining first-order equations have conserved translation and rotation quantities.  In the resulting radial variables they become the critical-point equations of an explicit function on a convex domain.  We prove that this function is strictly concave by factoring the relevant principal minors of its local Hessian.  Compactness gives existence, while strict concavity gives uniqueness.  The proof is analytic.  The symbolic scripts supplied with the paper check algebraic identities but are not used as part of the proof.  Combined with the classical odd-order and low-order results, this determines the maximal area for every $n\geq3$.
\end{abstract}

\section{Introduction}\label{sec:intro}

A planar polygon is called \emph{small} if the distance between every two of its vertices is at most one.  For a fixed number $n$ of vertices, let
\[
 A_n:=\sup\{\area(P): P\text{ is a simple small }n\text{-gon}\}.
\]
This is the polygonal form of the isodiametric area problem.  When $n$ is odd, Reinhardt proved that the regular small $n$-gon is optimal \cite{Reinhardt1922}, with
\begin{equation}\label{eq:oddvalue}
 A_n=\frac{n\sin(2\pi/n)}{8\cos^2(\pi/(2n))}\qquad(n\text{ odd}).
\end{equation}
The even case behaves differently.  The regular hexagon is already suboptimal.  Graham determined the optimum for $n=6$ \cite{Graham1975}, and Audet, Hansen, Messine and Xiong settled $n=8$ \cite{Audet2002}.  Subsequent work developed strong constructions, global-optimization formulations, and bounds for further even orders \cite{HenrionMessine2013,Mossinghoff2006,Bingane2023OL,Bingane2023DCG}.  Those results also show why the remaining problem is not merely numerical: good candidates are readily produced, but a proof must exclude all other diameter graphs and all nonsymmetric perturbations.

The main result is the following.

\begin{theorem}[Even orders]\label{thm:even}
For every even integer $n\geq8$, there is a unique maximum-area small $n$-gon up to Euclidean motions and reflection.
\end{theorem}

The proof starts from the diameter graph of a maximizing polygon.  A theorem of Foster and Szab\'o shows that, for even $n$, this graph consists of an odd cycle on $n-1$ vertices and one pendant diameter.  Lemma~\ref{lem:order} determines the only boundary order compatible with these diameters.  Figure~\ref{fig:structure} shows the resulting configuration for $n=8$; the solid curve is the polygonal boundary and the dashed segments are the diameter edges.

\begin{figure}[t]
\centering
\includegraphics[width=.62\textwidth]{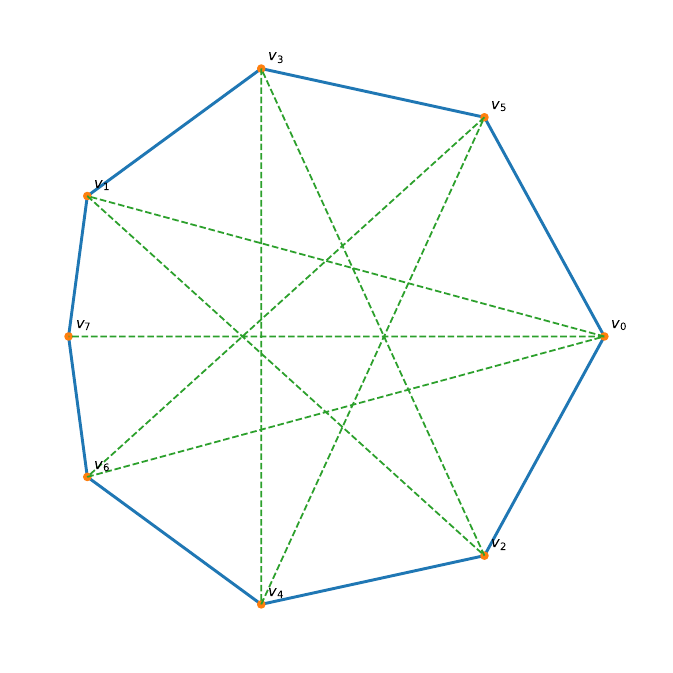}
\caption{The maximizing octagon reconstructed from the critical equations.  The dashed segments form the $7$-cycle of diameters together with the pendant diameter.  The labels refer to the diameter-cycle notation used in Section~\ref{sec:structure}.}
\label{fig:structure}
\end{figure}

After the combinatorial reduction, the argument is geometric.  The odd diameter cycle is the set of centers of a Reuleaux polygon.  The pendant vertex lies on one distinguished circular arc, and its optimal position can be eliminated explicitly.  The area problem is then an equality-constrained problem on the unit-distance cycle.  Translation and rotation invariance produce two conserved quantities for its first-order equations.  Once the configuration is recentered using those quantities, the stationary equations can be written as the gradient equations of a scalar function $\Phi_N$ on a convex domain of radial variables.  We then prove that $\Phi_N$ is strictly concave.  This rules out two distinct maximizing configurations.  Existence is supplied independently by compactness, so the concavity argument is used only for uniqueness.

The auxiliary function $\Phi_N$ is not the polygonal area, and its domain is larger than the set of globally realizable small polygons.  This distinction matters: the critical point used below is obtained from an actual maximizer, and strict concavity is then used only to exclude a second maximizing configuration.

\paragraph{Related work at $n=14$.}
The author has a separate manuscript that proves the case $n=14$ by a certified computer-assisted global argument.  That project predates the uniform proof given here and addresses the single case $n=14$ with a different method.  The overlap is confined to the preliminary compactness and diameter-graph reduction, including the crossing and boundary-order lemmas.  The Reuleaux reduction, the conserved quantities, the variational function $\Phi_N$, the strict-concavity proof, and the uniform argument for all even $n\geq8$ are specific to the present paper.  No numerical certificate from the $n=14$ manuscript is used below.

Section~\ref{sec:structure} establishes the diameter-graph and boundary-order reduction.  Section~\ref{sec:reuleaux} treats the Reuleaux representation and removes the pendant vertex.  Section~\ref{sec:noether} derives the first-order invariants.  Sections~\ref{sec:action} and \ref{sec:concavity} introduce the variational formulation and prove strict concavity.  Section~\ref{sec:global} reconstructs the polygon and proves Theorem~\ref{thm:even}.  Numerical values and drawings in Section~\ref{sec:numerics} are illustrations only.
\section{Structural reduction for even $n$}\label{sec:structure}

For a finite ordered tuple $X=(x_0,\ldots,x_{n-1})$, let $[X]=\conv\{x_0,\ldots,x_{n-1}\}$.  We first pass from arbitrary simple polygons to convex maximizers.

\begin{lemma}[Convex-hull reduction]\label{lem:hull}
For every simple $n$-gon $P$ with vertex set $\{p_0,\ldots,p_{n-1}\}$,
\[
 \area(P)\leq \area([P]),\qquad
 \diameter([P])=\max_{i,j}|p_i-p_j|.
\]
\end{lemma}

\begin{proof}
The bounded Jordan region of $P$ lies in its convex hull.  For fixed $y$, the convex function $x\mapsto |x-y|$ attains its maximum over a compact polytope at a vertex.  Applying the same observation to $y$ proves the diameter identity.
\end{proof}

Normalize one entry to the origin and define the compact-hull value
\[
 \widehat A_n:=\max\bigl\{\area([X]):x_0=0,\ |x_i-x_j|\leq1\ \text{for all }i,j\bigr\}.
\]
The displayed feasible set is compact and hull area is continuous, so the maximum exists.

\begin{lemma}[Outward bump]\label{lem:bump}
Let $X\subset\R^2$ be finite, of diameter at most one, and suppose $K=\conv X$ has positive area.  For every side $[a,b]$ of $K$ there is $q\notin K$ such that $\diameter(X\cup\{q\})\leq1$ and $\area(\conv(X\cup\{q\}))>\area(K)$.  The point can be chosen so that the only boundary change replaces $[a,b]$ by $[a,q]\cup[q,b]$.
\end{lemma}

\begin{proof}
Put $c=(a+b)/2$ and $L=|a-b|>0$.  For $v\in X$, the midpoint identity gives
\[
 |c-v|^2=\frac{|a-v|^2+|b-v|^2}{2}-\frac{L^2}{4}
 \leq 1-\frac{L^2}{4}<1.
\]
Because $X$ is finite, these inequalities have a common positive margin.  If $u$ is the exterior unit normal to $[a,b]$, then $q=c+\eps u$ has distance at most one from all old points for sufficiently small $\eps>0$.  The strict supporting-half-plane inequalities for the other sides persist.  The area increases by $L\eps/2$.
\end{proof}

\begin{proposition}[Attainment, strict monotonicity, and extremality]\label{prop:attain}
For every $n\geq3$,
\[
 \widehat A_{n+1}>\widehat A_n,\qquad A_n=\widehat A_n.
\]
Every maximizing $n$-tuple consists of $n$ distinct extreme points, has diameter one, and, in convex-boundary order, is a strictly convex maximum-area small $n$-gon.
\end{proposition}

\begin{proof}
Apply Lemma~\ref{lem:bump} to a maximizing tuple and append the bump point; this proves strict monotonicity.  If a maximizing $n$-tuple had only $h<n$ extreme points, then its positive-area hull would be feasible at order $h$ and would have area at most $\widehat A_h<\widehat A_n$, a contradiction.  Thus every entry is a distinct extreme point.  Listing them in boundary order gives a feasible strictly convex polygon, while Lemma~\ref{lem:hull} bounds every simple competitor by its hull.  Hence $A_n=\widehat A_n$.  Finally, a maximizing hull of diameter $d<1$ could be dilated by $1/d$, increasing area.
\end{proof}

\begin{theorem}[Foster--Szab\'o \cite{FosterSzabo2007}]\label{thm:FS}
If a unit-diameter $2m$-gon, $m\geq3$, has maximum area, then its complete diameter graph is a cycle on $2m-1$ vertices together with one pendant edge incident with the remaining vertex.
\end{theorem}

Let $n=2m$, set
\[
 N=2m-1=2k+1,
\]
and label the diameter cycle $v_0v_1\cdots v_{N-1}v_0$ and the pendant edge $v_0v_N$.  Put $d_i=v_{i+1}-v_i$, with cycle indices modulo $N$; thus $|d_i|=1$.

\begin{lemma}[Crossing diameters]\label{lem:crossing}
Two disjoint diameter segments of a finite planar set cross properly.
\end{lemma}

\begin{proof}
If $p,q$ are endpoints of a diameter and $r$ is another point, expansion of $|r-q|^2\leq|p-q|^2$ gives
\[
 (p-q)\cdot(r-p)\leq-\frac12|r-p|^2<0.
\]
Thus each endpoint is strictly exposed.  Four distinct endpoints of two diameters form a strict convex quadrilateral.  If the diameter segments were opposite sides rather than diagonals, the strict triangle inequalities at the intersection of the two diagonals would imply that the sum of those two diameter sides is strictly smaller than the sum of the diagonals, which is at most the same value.  This is impossible.
\end{proof}

\begin{lemma}[Canonical boundary order]\label{lem:order}
Up to reversing both the cycle labels and the boundary orientation, the convex-boundary order is
\begin{equation}\label{eq:boundaryorder}
 v_0,v_{N-2},v_{N-4},\ldots,v_1,v_N,v_{N-1},v_{N-3},\ldots,v_2.
\end{equation}
The pendant occupies the unique gap between $v_1$ and $v_{N-1}$.
\end{lemma}

\begin{proof}
Remove $v_N$ temporarily, and let $p_i\in\mathbb Z/N\mathbb Z$ be the boundary position of $v_i$.  The cycle edges disjoint from $v_iv_{i+1}$ form a path on the other $N-2$ vertices.  By Lemma~\ref{lem:crossing}, successive vertices of this path alternate between the two open boundary arcs cut by $v_i,v_{i+1}$.  Their populations are $m-1$ and $m-2$.  Therefore
\[
 p_{i+1}-p_i=\eps_i(m-1)\pmod N,\qquad \eps_i\in\{-1,1\}.
\]
Summation around the odd cycle yields $(m-1)\sum_i\eps_i\equiv0\pmod N$.  Since $\gcd(m-1,N)=1$, the odd integer $\sum_i\eps_i$, whose absolute value is at most $N$, equals $N$ or $-N$.  All signs are equal.  Taking $p_0=0$ and choosing orientation gives $p_i=(m-1)i$.  As $(m-1)^{-1}\equiv-2\pmod N$, the cycle labels in boundary order are those in \eqref{eq:boundaryorder} with $v_N$ omitted.

The pendant diameter crosses every edge of the path $v_1,\ldots,v_{N-1}$.  Hence the path vertices again alternate across its two boundary arcs, with $m-1$ vertices on each.  In the order just found, the odd and even labels form two contiguous blocks.  The only nonempty admissible gap between them not incident with $v_0$ is $v_1\mid v_{N-1}$.
\end{proof}

Define the positive star angles by
\begin{equation}\label{eq:starangle}
 d_{i-1}\cdot d_i=-\cos\gamma_i,\qquad
 d_{i-1}\cross d_i=\sin\gamma_i.
\end{equation}
Every chord $v_{i-1}v_{i+1}$ has length $2\sin(\gamma_i/2)\leq1$.  Since Theorem~\ref{thm:FS} gives the complete diameter graph, equality would be an extra diameter, so
\begin{equation}\label{eq:anglebound}
 0<\gamma_i<\frac{\pi}{3}.
\end{equation}
The step circuit in the convex order has rotation number $m-1$: its directed turns are $\pi-\gamma_i$ and sum to $2\pi(m-1)$.  Consequently
\begin{equation}\label{eq:turningsum}
 \sum_{i=0}^{N-1}\gamma_i=\pi.
\end{equation}
One way to verify the rotation number is by continuously deforming the strictly convex boundary to a circle.  The direction of its $q$-step chord, $q=m-1$, winds once as the boundary parameter winds once; traversing the vertices in $q$-steps counts each positive gap between consecutive chord directions exactly $q$ times.

\section{Reuleaux representation and the pendant vertex}\label{sec:reuleaux}

Write the cycle vertices in convex-boundary order as
\begin{equation}\label{eq:xmap}
 x_j=v_{-2j\pmod N},\qquad j=0,\ldots,N-1.
\end{equation}
Since $N=2k+1$,
\[
 x_{j+k}=v_{-2j+1},\qquad x_{j+k+1}=v_{-2j-1},
\]
and both are at unit distance from $x_j$.  Together with $|x_i-x_j|\leq1$, these are precisely the standard vertex constraints for a Reuleaux $N$-gon; compare the variational vertex formulation in \cite{Bogosel2025}.  Explicitly, let
\[
 R=\bigcap_{j=0}^{N-1}B(x_j,1).
\]
The boundary arc centered at $x_j$ joins $x_{j+k}$ to $x_{j+k+1}$.  The centers are distinct and each arc has positive length.  At $v_i$ its unit endpoint radii are $-d_{i-1}$ and $d_i$, so its angular length is exactly $\gamma_i$.

The marked arc centered at $v_0=x_0$ has endpoints $v_1$ and $v_{N-1}$.  The pendant vertex lies in their boundary gap, on this arc.  Put
\[
 \gamma_0=2\theta,\qquad 0<\theta<\frac{\pi}{6}.
\]
For a unit circular segment of angle $z$, set
\[
 s(z)=\frac{z-\sin z}{2}.
\]
If the pendant divides the marked arc into angles $t$ and $2\theta-t$, polygonization loses $s(t)+s(2\theta-t)$.  Since $s''(z)=\sin z/2>0$ on $(0,\pi)$, the loss is uniquely minimized at $t=\theta$.  Thus the pendant is the midpoint of the marked arc and
\begin{equation}\label{eq:pendantbonus}
 \area(P)=A_S+p(\theta),
 \qquad p(\theta)=s(2\theta)-2s(\theta)
 =\sin\theta(1-\cos\theta),
\end{equation}
where $A_S$ is the area of the polygonal Reuleaux skeleton.  In diameter-cycle order and the chosen orientation,
\begin{equation}\label{eq:skeletonarea}
 A_S=\frac12\sum_{i=0}^{N-1}v_i\cross v_{i-2}.
\end{equation}

The next lemma justifies differentiating the reduced expression; midpoint optimality alone would not suffice.

\begin{lemma}[Local elimination of the pendant vertex]\label{lem:envelope}
Let $(v^*,w^*)$ be a global maximizer with the structure above, and let
\[
 \mathcal M=\{v:g_i(v)=\tfrac12(|v_{i+1}-v_i|^2-1)=0,\ i\in\mathbb Z/N\mathbb Z\}.
\]
In a neighborhood of $v^*$ in $\mathcal M$, the midpoint $w(v)$ of the marked shorter arc is smooth, $(v,w(v))$ remains a feasible small strictly convex polygon, and its area is $A_S(v)+p(\theta(v))$.  Therefore $v^*$ is a local maximizer of the reduced objective on $\mathcal M$.
\end{lemma}

\begin{proof}
At the maximizer the full diameter graph is exactly the cycle plus the pendant edge.  Every other distance is strictly below one; finitely many strict inequalities have a common continuity margin.  Strict convex order of the full polygon and positive arc lengths also persist under small perturbations.  The midpoint of the nonzero marked shorter arc is the smooth function
\[
 w(v)=v_0+\frac{(v_1-v_0)+(v_{N-1}-v_0)}
 {|(v_1-v_0)+(v_{N-1}-v_0)|}.
\]
Its nonactive distances remain below one by continuity.  Thus every sufficiently close point of the equality manifold has a feasible midpoint completion, with the same boundary order and area \eqref{eq:pendantbonus}.  Global maximality implies the claimed local maximality.
\end{proof}

\begin{lemma}[Independence of the cycle constraints]\label{lem:licq}
The gradients of the $N$ cycle constraints $g_i$ are linearly independent at every configuration above.
\end{lemma}

\begin{proof}
If $\sum_i a_i\nabla g_i=0$, then at $v_i$
\[
 a_{i-1}d_{i-1}-a_i d_i=0.
\]
Consecutive cycle diameters are noncollinear by \eqref{eq:anglebound}, hence $a_{i-1}=a_i=0$.  This holds for every $i$.
\end{proof}

Lemma~\ref{lem:envelope} gives a smooth reduced problem near a maximizer.  By Lemma~\ref{lem:licq}, its equality constraints are independent there, so the usual equality-constrained KKT conditions apply.

\section{First-order equations and conserved quantities}\label{sec:noether}

Let $J(x_1,x_2)=(-x_2,x_1)$, so $(Jx)\cdot y=x\cross y$.  Differentiating \eqref{eq:skeletonarea} gives
\begin{equation}\label{eq:ASgrad}
 \nabla_{v_i}A_S=\frac12J(d_{i-2}+d_{i-1}+d_i+d_{i+1}).
\end{equation}
Moreover
\[
 p'(\theta)=(1-\cos\theta)(1+2\cos\theta),\qquad
 \eta=\frac{p'(\theta)}{2\sin(2\theta)}>0.
\]
On the cycle equality manifold we may extend the marked angle locally by
\[
 \widetilde\theta=\frac12\arccos(-d_{-1}\cdot d_0).
\]
It agrees there with the geometric angle defined using normalized edge directions.  The gradients of the two extensions differ only by a linear combination of $\nabla g_{-1}$ and $\nabla g_0$, which is absorbed by shifting the neighboring cycle multipliers.  Differentiating this gauge extension gives the KKT-equivalent pendant forces
\begin{equation}\label{eq:forces}
 f_{-1}=-\eta d_0,\qquad
 f_0=\eta(d_0-d_{-1}),\qquad
 f_1=\eta d_{-1},\qquad f_i=0\ \text{otherwise}.
\end{equation}
They have zero total force and torque.  With multipliers $\lambda_i$, KKT reads
\begin{equation}\label{eq:kkt}
 \frac12J(d_{i-2}+d_{i-1}+d_i+d_{i+1})
 +\lambda_{i-1}d_{i-1}-\lambda_i d_i+f_i=0.
\end{equation}

At a bulk index, projection of two neighboring equations gives
\begin{equation}\label{eq:bulk-lambda}
 2\lambda_i\sin\gamma_i
 =1-\cos\gamma_{i-1}-\cos\gamma_i
  +\cos(\gamma_i+\gamma_{i+1}).
\end{equation}
The four edges $-2,-1,0,1$ have the corrections caused by \eqref{eq:forces}; the raw forms used below are recorded in the proof of Lemma~\ref{lem:branch}.  We will also use the seven residue classes
\[
 -3,-2,-1,0,1,2,3,
\]
which are distinct because $N\geq7$ (equivalently, $n\geq8$).

Define the discrete translation momentum
\begin{equation}\label{eq:momentum}
 P_i=\frac12J(v_{i-1}+v_i+v_{i+1}+v_{i+2})-\lambda_i d_i.
\end{equation}
Subtracting consecutive terms and using \eqref{eq:kkt} yields
\begin{equation}\label{eq:momentumjump}
 P_i-P_{i-1}=-f_i.
\end{equation}
Thus $P_i$ is constant in the bulk.  In particular,
\begin{equation}\label{eq:Pjumps}
 P_{-2}=P_1=:P,\qquad
 P_{-1}=P+\eta d_0,\qquad P_0=P+\eta d_{-1}.
\end{equation}
A translation by $q$ changes $P$ to $P+2Jq$, so a unique translation makes the bulk momentum zero.  We call these \emph{momentum-centered coordinates}.

For rotation, define
\begin{equation}\label{eq:angularmomentum}
 M_i=v_i\cross P_i-|v_i|^2-\frac12d_{i-1}\cdot d_i.
\end{equation}
In the bulk, direct subtraction gives $M_i=M_{i-1}$.  Under translation the invariant scalar is
\[
 C=M-\frac14|P|^2.
\]
Put $u_i=|v_i|^2$ in momentum-centered coordinates.  The bulk identity and the two jump cancellations in \eqref{eq:Pjumps} give
\begin{equation}\label{eq:noether-relations}
 u_i+C=\frac12\cos\gamma_i\quad(i\neq0),
\end{equation}
Direct subtraction in \eqref{eq:angularmomentum} gives
\[
 M_i-M_{i-1}=v_i\cross(P_i-P_{i-1})=-v_i\cross f_i.
\]
Taking the bulk value to be $C$ in momentum-centered coordinates, the jump at $-1$ cancels $v_{-1}\cross P_{-1}$, while at $0$
\[
 M_0-v_0\cross P_0=C-\eta(d_{-1}\cross d_0)
 =C-\frac12p'(\theta).
\]
The third jump restores the bulk value because the total pendant torque is zero.  Therefore, at the marked vertex,
\begin{equation}\label{eq:marked-noether}
 u_0+C=\frac12\cos(2\theta)+\frac12p'(\theta)
 =\frac12\cos\theta.
\end{equation}
The last equality follows from $\cos(2\theta)+p'(\theta)=\cos\theta$.  Thus the marked index has the modified identity \eqref{eq:marked-noether}, rather than the bulk relation \eqref{eq:noether-relations}.

For the oriented unit edge $d_i$, define the signed altitude of the centered origin by
\[
 h_i=d_i\cross(-v_i).
\]

\begin{lemma}[Sign of the center altitudes]\label{lem:branch}
Every global maximizer satisfies $h_i>0$ for all $i$.
\end{lemma}

\begin{proof}
For a bulk edge, \eqref{eq:momentum} with $P_i=0$ gives
\begin{equation}\label{eq:hbulk}
 h_i=\frac{\sin\gamma_i+\sin\gamma_{i+1}}4+\frac{\lambda_i}{2}.
\end{equation}
Let $a,b,c,d\in(0,\pi/2)$ be four consecutive star angles and let $\lambda$ be the multiplier between $b,c$.  The two bulk projections are
\[
 2\lambda\sin b=1-\cos a-\cos b+\cos(b+c),
\]
\[
 2\lambda\sin c=1-\cos d-\cos c+\cos(b+c).
\]
If $h\leq0$ and $b\geq c$, the first equation implies
\[
 \cos a\geq F(b,c):=2-\cos b+\cos b\cos c-\cos^2b>1,
\]
because, with $x=\cos b\leq y=\cos c$, one has
$F(b,c)-1=(1-x)+x(y-x)>0$.  If $c\geq b$, the second equation gives the symmetric contradiction $\cos d>1$.

Four edges touch the pendant stencil.  For $\lambda_{-2}$, write
$b=\gamma_{-2}$, $r=\gamma_{-1}$, and $c=2\theta$.  Its left representation is the bulk one, while the right representation is
\begin{equation}\label{eq:lminus2}
 2\lambda_{-2}\sin r
 =1-\cos c-\cos r+\cos(b+r)-p'(\theta).
\end{equation}
If $b\geq r$, the preceding bulk contradiction applies.  If $r\geq b$ and $h_{-2}\leq0$, then
\[
 \cos c\geq F(r,b)-p'(\theta)>\cos c,
\]
because
\begin{align*}
F(r,b)-p'(\theta)-\cos c
 &= [F(r,b)-1]+[1-\cos(2\theta)-p'(\theta)]\\
 &= [F(r,b)-1]+1-\cos\theta>0.
\end{align*}
The edge $1$ follows from the same algebra after reversing the local index order; no symmetry of the unknown polygon is assumed.

For the central edge $-1$, \eqref{eq:Pjumps} gives
\[
 4h_{-1}=\sin r+\sin c+2\lambda_{-1}-2\eta\cos c.
\]
Its two projected equations are
\begin{align*}
2\lambda_{-1}\sin r
 &=1-\cos b-\cos r+\cos(r+c)+2\eta\sin(r+c),\\
2\lambda_{-1}\sin c
 &=1-\cos d-\cos c+\cos(r+c)+p'(\theta).
\end{align*}
If $h_{-1}\leq0$, the first equation gives $\cos b\geq F(r,c)+2\eta\cos r\sin c>1$ when $r\geq c$, and the second gives $\cos d\geq F(c,r)+p'(\theta)(1-\cos c)>1$ when $c\geq r$.  Both are impossible.  The edge $0$ follows by the same index-reversal calculation, again without a symmetry assumption.  All angles used here lie in $(0,\pi/2)$ by \eqref{eq:anglebound} and $\theta<\pi/6$.
\end{proof}

The positivity in Lemma~\ref{lem:branch} makes each triangle with vertices $0,v_i,v_{i+1}$ nondegenerate and gives all of them the same orientation.

\section{Variational formulation}\label{sec:action}

In the variables
\[
 X=(C,u_0,\ldots,u_{N-1})\in\R^{N+1},
\]
define
\begin{equation}\label{eq:anglesfromX}
 \theta=\arccos(2(u_0+C)),\qquad
 \gamma_0=2\theta,\qquad
 \gamma_i=\arccos(2(u_i+C))\quad(i\neq0).
\end{equation}
Define $\Dnat$ by the cyclic conditions
\begin{equation}\label{eq:Dnat}
 u_i>0,\qquad 0<2(u_i+C)<1,\qquad
 |\sqrt{u_i}-\sqrt{u_{i+1}}|<1<\sqrt{u_i}+\sqrt{u_{i+1}}.
\end{equation}
Every global maximizer gives a point of $\Dnat$.  The angle inequalities follow from Section~\ref{sec:structure}, while Lemma~\ref{lem:branch} gives positive, nondegenerate triangles with the centered origin and hence the strict triangle inequalities.

\begin{lemma}[Convexity of the triangle domain]\label{lem:Dconvex}
$\Dnat$ is open and convex.
\end{lemma}

\begin{proof}
The first two groups in \eqref{eq:Dnat} are open half-spaces.  On the positive quadrant, $\sqrt u+\sqrt v$ is concave, so its strict superlevel set $\sqrt u+\sqrt v>1$ is convex.  The other triangle inequality is
\[
 g(u,v):=(\sqrt u-\sqrt v)^2=u+v-2\sqrt{uv}<1.
\]
The geometric mean is concave; equivalently, its Hessian has nonpositive trace and zero determinant.  Thus $g$ is convex and its strict sublevel set is convex.  Intersecting these sets proves the claim.
\end{proof}

For a triangle with side lengths $\sqrt u,\sqrt v,1$, let
\begin{equation}\label{eq:Delta}
 \Delta(u,v)=\frac14\sqrt{2uv+2u+2v-u^2-v^2-1}
\end{equation}
be its area.  Denote by $\alpha$ and $\beta$ its base angles at the vertices of radial lengths $\sqrt u$ and $\sqrt v$.  Define
\begin{equation}\label{eq:E}
 E(u,v)=4\Delta-2u\alpha-2v\beta.
\end{equation}
Elementary differentiation of the cosine law gives
\begin{equation}\label{eq:Ederivatives}
 E_u=-2\alpha,\quad E_v=-2\beta,\quad E_{uv}=-\frac1{2\Delta},
\end{equation}
\[
 E_{uu}=\frac{u+v-1}{4u\Delta},\qquad
 E_{vv}=\frac{u+v-1}{4v\Delta}.
\]
For $x\in(0,\pi/2)$ determined by $\cos x=2(u+C)$, set
\begin{equation}\label{eq:V}
 V_C(u)=x\cos x-\sin x.
\end{equation}
Then
\begin{equation}\label{eq:Vderivatives}
 (V_C)_u=(V_C)_C=2x,\qquad
 (V_C)_{uu}=(V_C)_{uC}=(V_C)_{CC}=-\frac4{\sin x}.
\end{equation}

Define
\begin{equation}\label{eq:Phi}
 \boxed{\quad
 \Phi_N(X)=\sum_{i=0}^{N-1}E(u_i,u_{i+1})
 +2V_C(u_0)+\sum_{i=1}^{N-1}V_C(u_i)-2\pi C.
 \quad}
\end{equation}
The function $\Phi_N$ is an auxiliary variational object; it is not the polygonal area.

Let $\alpha_i,\beta_i$ be the base angles in the triangle $(0,v_i,v_{i+1})$.  Positivity of the altitudes gives
\begin{equation}\label{eq:anglematching}
 \gamma_i=\beta_{i-1}+\alpha_i\quad(i\neq0),\qquad
 2\theta=\beta_{N-1}+\alpha_0.
\end{equation}
Using \eqref{eq:Ederivatives}--\eqref{eq:Vderivatives},
\begin{align}
 \frac{\partial\Phi_N}{\partial u_i}
 &=2(\gamma_i-\beta_{i-1}-\alpha_i),&&i\neq0,\label{eq:gradui}\\
 \frac{\partial\Phi_N}{\partial u_0}
 &=2(2\theta-\beta_{N-1}-\alpha_0),\label{eq:gradu0}\\
 \frac{\partial\Phi_N}{\partial C}
 &=2\left(2\theta+\sum_{i=1}^{N-1}\gamma_i-\pi\right).
 \label{eq:gradC}
\end{align}
Equations \eqref{eq:anglematching} and \eqref{eq:turningsum} therefore give
\begin{equation}\label{eq:bridge}
 \text{reduced KKT and positive-altitude branch}\quad\Longrightarrow\quad
 \nabla\Phi_N(X)=0.
\end{equation}

\section{Strict concavity}\label{sec:concavity}

For adjacent variables define the local block
\begin{equation}\label{eq:L}
 L(C,u,v)=E(u,v)+\frac12V_C(u)+\frac12V_C(v).
\end{equation}
Let $x,y\in(0,\pi/2)$ satisfy $\cos x=2(u+C)$ and $\cos y=2(v+C)$, and put $s_x=\sin x$, $s_y=\sin y$.  From \eqref{eq:Ederivatives} and \eqref{eq:Vderivatives}, the matrix $K=-D^2L$ in variables $(C,u,v)$ is
\begin{equation}\label{eq:K}
 K=\begin{pmatrix}
 \frac2{s_x}+\frac2{s_y}&\frac2{s_x}&\frac2{s_y}\\[2mm]
 \frac2{s_x}&\frac{1-u-v}{4u\Delta}+\frac2{s_x}&\frac1{2\Delta}\\[2mm]
 \frac2{s_y}&\frac1{2\Delta}&\frac{1-u-v}{4v\Delta}+\frac2{s_y}
 \end{pmatrix}.
\end{equation}

The first principal minor is positive.  For the remaining two, put
\[
 p=\frac{x+y}{2},\quad q=\frac{|x-y|}{2},\quad
 a=\sin p,\quad b=\cos q,\quad c=\sin q,
\]
\[
 d=|u-v|=ac,\qquad S=u+v,\qquad t=4\Delta.
\]
The identity $d=ac$ follows by subtracting the two cosine relations; Heron's identity and the sine addition formula give
\begin{equation}\label{eq:heronhalf}
 t^2=2S-d^2-1,\qquad s_x+s_y=2ab.
\end{equation}

Direct expansion of \eqref{eq:K} gives
\begin{equation}\label{eq:detK}
 \det K=\frac{S-d^2-tab}{\Delta s_xs_yuv}.
\end{equation}
The numerator $Q=S-d^2-tab$ has the exact factorization
\begin{equation}\label{eq:Qfactor}
 \boxed{\quad 2Q=(t-ab)^2+\cos^2p>0.\quad}
\end{equation}
This follows after substituting $S=(t^2+d^2+1)/2$, $d=ac$, and $b^2+c^2=1$.

For the second principal minor, permute $u,v$ if needed.  Take the second coordinate to be
\[
 w=\max(u,v)=\frac{S+d}{2}.
\]
Its positive denominator may be suppressed; the numerator is
\begin{equation}\label{eq:Rraw}
 R=t(S+d)+2ab(1-S).
\end{equation}
Using only the first relation in \eqref{eq:heronhalf},
\begin{equation}\label{eq:Rfactor}
 \boxed{\quad
 R=ab(1-d^2)+t^2(1+d-ab)+\frac t2(t-1-d)^2>0.
 \quad}
\end{equation}
Here $0<p<\pi/2$, $0\leq q<\pi/4$, so $a,b,t>0$ and $0\leq d<1$.  Moreover
\[
 1+d-ab=1-a(b-c)>0,
\]
because $a<1$ and $0<b-c\leq1$.  The selected principal minors are therefore positive, and Sylvester's criterion yields
\begin{equation}\label{eq:localconcave}
 D^2L(C,u,v)\prec0
\end{equation}
throughout the domain defined by the triangle inequalities in \eqref{eq:Dnat}.

Rewrite the action as
\begin{equation}\label{eq:Phiblocks}
 \Phi_N=\sum_{i=0}^{N-1}L(C,u_i,u_{i+1})+V_C(u_0)-2\pi C.
\end{equation}
Each embedded edge Hessian is negative semidefinite in the global variables and strictly negative on a variation for which $(\delta C,\delta u_i,\delta u_{i+1})\neq0$.  Every nonzero global variation activates at least one such triple.  The extra $V_C(u_0)$ has negative-semidefinite Hessian, and the last term is linear.  Hence
\begin{theorem}[Concavity of the variational function]\label{thm:concavity}
For every odd $N\geq7$,
\[
 D^2\Phi_N(X)\prec0\qquad(X\in\Dnat).
\]
Consequently $\Phi_N$ has at most one critical point in the convex domain $\Dnat$.
\end{theorem}

\section{Reconstruction and uniqueness}\label{sec:global}

Proposition~\ref{prop:attain} supplies at least one global maximizer.  Sections~\ref{sec:structure}--\ref{sec:action} map every such maximizer to a point $X_{\max}\in\Dnat$ satisfying $\nabla\Phi_N(X_{\max})=0$.  By Theorem~\ref{thm:concavity}, all global maximizers have the same radial data.  It remains to show that these data determine the polygon.

\begin{lemma}[Reconstruction]\label{lem:injective}
Two stationary marked-Reuleaux configurations on the positive-altitude branch with the same labeled radial data $X$ differ by an orientation-preserving Euclidean motion.  If orientation is forgotten, the only additional ambiguity is reflection.
\end{lemma}

\begin{proof}
Equations \eqref{eq:anglesfromX} recover all star angles.  After fixing $v_0=0$ and $d_0=(1,0)$, the chosen orientation gives recursively
\[
 d_i=R_{\pi-\gamma_i}d_{i-1},\qquad v_{i+1}=v_i+d_i.
\]
Thus an already closed cycle is unique under the normalization.  The pendant is the unique midpoint of the marked arc centered at $v_0$.  Reversing orientation reflects the configuration.
\end{proof}

To reconstruct directly from the critical point, put $r_i=\sqrt{u_i}$.  For each cyclic pair, \eqref{eq:Dnat} gives a unique nondegenerate triangle with sides $r_i,r_{i+1},1$.  Let $\delta_i\in(0,\pi)$ be its angle at the centered origin.  Write the reconstructed cycle as $z_i$ (the notation avoids using $v_N$, which is reserved for the pendant vertex).  Normalize $O=0$, $z_0=(r_0,0)$, and set
\begin{equation}\label{eq:reconstruction}
 z_{i+1}=r_{i+1}R_{\delta_i}\frac{z_i}{r_i}.
\end{equation}
Every reconstructed cycle edge has length one.  At a critical point, \eqref{eq:gradui}--\eqref{eq:gradu0} imply
\[
 \sum_i\gamma_i=\sum_i(\alpha_i+\beta_i)
 =N\pi-\sum_i\delta_i.
\]
Equation \eqref{eq:gradC} gives $\sum_i\gamma_i=\pi$, so
\begin{equation}\label{eq:radialclosure}
 \sum_i\delta_i=(N-1)\pi=2k\pi.
\end{equation}
The radial direction returns after an integer number of turns and $r_N=r_0$, hence $z_N=z_0$: the cycle closes automatically.  The pendant is then inserted at the marked midpoint.

The domain $\Dnat$ records only the local triangle inequalities.  A point of $\Dnat$ need not reconstruct a small polygon, since non-neighbor distances are not among its defining inequalities.  Surjectivity is not needed.  The critical point used in the proof comes from an admissible global maximizer, so its reconstructed polygon is already known to satisfy the remaining distance constraints.

Combining compact existence, strict concavity, and Lemma~\ref{lem:injective} proves Theorem~\ref{thm:even}.  It also gives the promised exact characterization:
\begin{equation}\label{eq:exactvalue}
 A_{N+1}=\frac12\sum_i z_i(X_N)\cross z_{i-2}(X_N)
 +\sin\theta(X_N)\bigl(1-\cos\theta(X_N)\bigr),
\end{equation}
where $X_N$ is the unique critical point of \eqref{eq:Phi} in $\Dnat$ and the vertices are reconstructed by \eqref{eq:reconstruction}.

\section{Low orders and completion}\label{sec:completion}

Reinhardt's theorem gives \eqref{eq:oddvalue} for every odd $n\geq3$ \cite{Reinhardt1922}.  For $n=4$, the elementary diagonal-area inequality gives
\[
 A_4=\frac12;
\]
the optimizer is not unique.  Graham's theorem gives the exact largest small hexagon and $A_6=0.674981\ldots$ \cite{Graham1975}.  For $N=5$ some of the seven indices used in the pendant calculation coincide, so the uniform argument does not include the hexagon.

\begin{corollary}[All orders]\label{cor:complete}
The value $A_n$ is determined for every $n\geq3$: by \eqref{eq:oddvalue} for odd $n$, by the classical results for $n=4,6$, and by \eqref{eq:exactvalue} for even $n\geq8$.
\end{corollary}

\section{Numerical examples and checks}\label{sec:numerics}

Solving the critical equations numerically and reconstructing the cycle gives the values in Table~\ref{tab:values}.  They are not used in any implication above.

\begin{table}[h]
\caption{Numerical evaluations of the analytically characterized maxima}
\label{tab:values}
\centering
\begin{tabular}{cc}
\toprule
$n$ & $A_n$ \\
\midrule
8  & 0.7268684827516267 \\
10 & 0.7491373458778303 \\
12 & 0.7607298734487959 \\
14 & 0.7675310111207523 \\
16 & 0.7718613219805711 \\
18 & 0.7747881650744942 \\
\bottomrule
\end{tabular}
\end{table}

\begin{figure}[t]
\centering
\begin{minipage}{.47\textwidth}
  \centering
  \includegraphics[width=\linewidth]{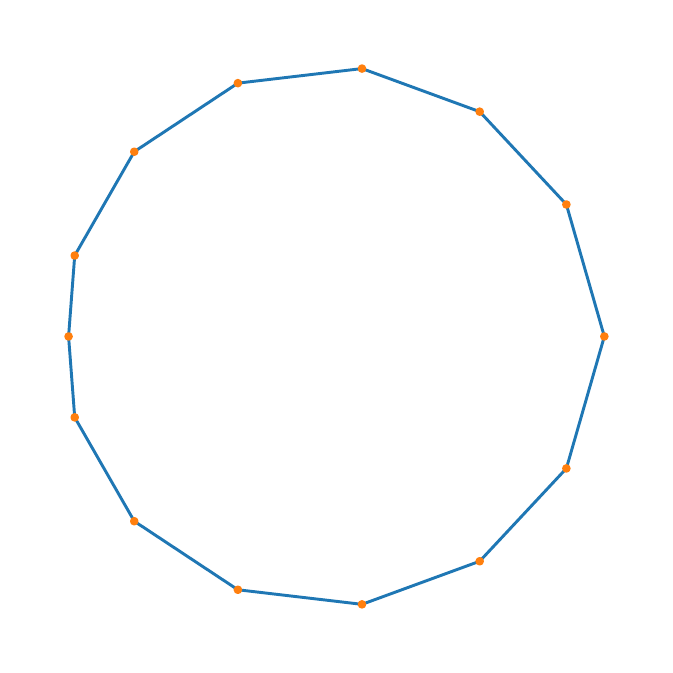}
\end{minipage}\hfill
\begin{minipage}{.47\textwidth}
  \centering
  \includegraphics[width=\linewidth]{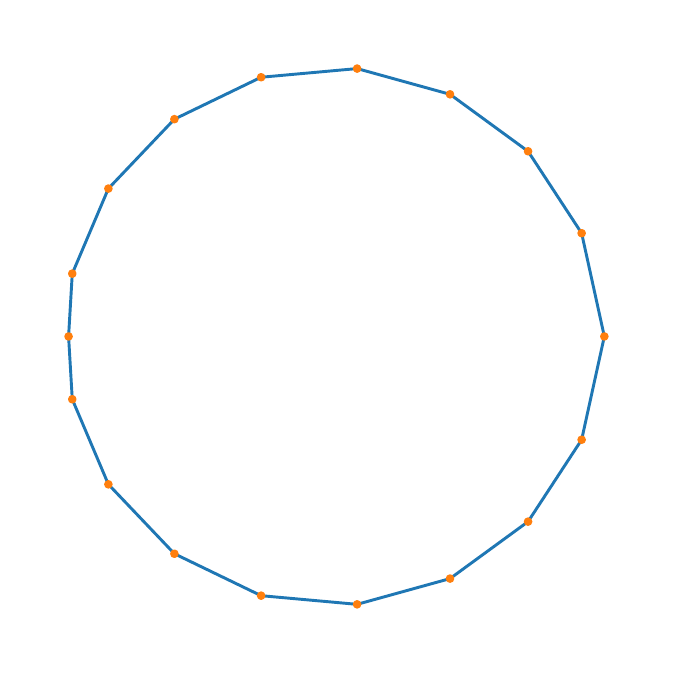}
\end{minipage}
\caption{Numerically reconstructed maximizers for $n=14$ (left) and $n=18$ (right).  The drawings are obtained from the critical equations and are not used in the proof.}
\label{fig:examples}
\end{figure}

At each listed solution, reconstruction closes to numerical precision, all pairwise distances are at most one, and the reflection-related variables agree.  The $n=14$ value also agrees with the independently certified value in the separate $n=14$ manuscript.  The accompanying reproducibility archive contains scripts that check \eqref{eq:forces}, the momentum and moment jumps, the action derivatives, and the factorizations \eqref{eq:Qfactor}--\eqref{eq:Rfactor}, and includes an independent nonlinear-real-arithmetic verifier.  These computations serve as consistency checks.  The proof is contained in the displayed arguments above.

\appendix

\section{Differential identities used in the proof}\label{app:calculations}

We collect here the differential identities used in Sections~\ref{sec:noether}--\ref{sec:concavity}.

\subsection{Skeleton area, gradient, and pendant gauge}

By \eqref{eq:xmap}, consecutive skeleton vertices in boundary order have cycle labels $i$ and $i-2$.  The shoelace formula therefore gives
\[
 A_S=\frac12\sum_i v_i\cross v_{i-2}.
\]
Only the terms with first label $i$ and second label $i$ contribute to the gradient at $v_i$, whence
\[
 \nabla_{v_i}A_S
 =\frac12J(v_{i+2}-v_{i-2})
 =\frac12J(d_{i-2}+d_{i-1}+d_i+d_{i+1}).
\]

Put $q=d_{-1}$, $r=d_0$, and let $c=2\theta$ be the geometric angle defined near the equality manifold by
\[
 c(q,r)=\arccos\left(-\frac{q\cdot r}{|q||r|}\right).
\]
At $|q|=|r|=1$, writing $s=q\cdot r$ gives
\[
 \dd c=\frac{\dd s-s(q\cdot\dd q+r\cdot\dd r)}{\sin c}.
\]
The gauge extension $\widetilde c=\arccos(-q\cdot r)$ instead has
\[
 \dd\widetilde c=\frac{\dd s}{\sin c}.
\]
Their difference is a linear combination of
$\dd g_{-1}=q\cdot\dd q$ and $\dd g_0=r\cdot\dd r$; changing from one extension to the other therefore only shifts $\lambda_{-1},\lambda_0$.  Since
\[
 \dd(q\cdot r)
 =-r\cdot\dd v_{-1}+(r-q)\cdot\dd v_0+q\cdot\dd v_1,
\]
we obtain
\[
 \dd p(\widetilde c/2)=\frac{p'(\theta)}{2\sin c}\dd(q\cdot r)
 =\eta\dd(q\cdot r),
\]
which gives \eqref{eq:forces}.  Translation and rotation invariance give zero total force and torque; directly,
\[
 f_{-1}+f_0+f_1=0,\qquad
 v_{-1}\cross f_{-1}+v_0\cross f_0+v_1\cross f_1=0.
\]

\subsection{Telescoping identities and the marked moment}

From \eqref{eq:momentum},
\begin{align*}
P_i-P_{i-1}
 &=\frac12J(v_{i+2}-v_{i-2})
   +\lambda_{i-1}d_{i-1}-\lambda_i d_i\\
 &=-f_i
\end{align*}
by \eqref{eq:kkt}.  Expanding \eqref{eq:angularmomentum}, using
$v_i=v_{i-1}+d_{i-1}$ and $|d_{i-1}|=|d_i|=1$, gives the exact identity
\[
 M_i-M_{i-1}=v_i\cross(P_i-P_{i-1}).
\]
In centered coordinates take the bulk values $P=0$ and $M=C$.  At the first jump,
\[
 M_{-1}=C+\eta v_{-1}\cross d_0,\qquad P_{-1}=\eta d_0,
\]
so $M_{-1}-v_{-1}\cross P_{-1}=C$.  At the marked index,
\begin{align*}
M_0-v_0\cross P_0
 &=C+\eta(v_{-1}-v_0)\cross d_0\\
 &=C-\eta(d_{-1}\cross d_0)
 =C-\frac12p'(\theta).
\end{align*}
Since
\[
 M_i-v_i\cross P_i=-u_i+\frac12\cos\gamma_i,
\]
this yields
\[
 u_0+C=\frac12\cos(2\theta)+\frac12p'(\theta)
 =\frac12\cos\theta.
\]
The last jump restores the bulk moment by the zero-torque identity above.

\subsection{Triangle Hessian and the two positive numerators}

For the triangle in \eqref{eq:Delta}, differentiation of the cosine laws
\[
 \cos\alpha=\frac{u+1-v}{2\sqrt u},\qquad
 \cos\beta=\frac{v+1-u}{2\sqrt v}
\]
and of Heron's expression gives \eqref{eq:Ederivatives}.  Differentiating
$x\cos x-\sin x$ subject to $\cos x=2(u+C)$ gives \eqref{eq:Vderivatives}.  Substitution produces the matrix \eqref{eq:K}.  Its $(C,u)$ principal minor and determinant simplify to
\[
 \det K_{C,u}
 =\frac{8\Delta u+(s_x+s_y)(1-u-v)}
 {2\Delta s_xs_yu},
\]
\[
 \det K
 =\frac{u+v-(u-v)^2-2\Delta(s_x+s_y)}
 {\Delta s_xs_yuv}.
\]
After placing $w=\max(u,v)$ in the second coordinate, their numerators are respectively $R$ and $Q$ from \eqref{eq:Rraw} and \eqref{eq:detK}.

The factorization of $Q$ follows directly from \eqref{eq:heronhalf}:
\begin{align*}
2Q&=t^2+1-d^2-2tab\\
  &=(t-ab)^2+1-a^2b^2-a^2c^2
   =(t-ab)^2+1-a^2.
\end{align*}
For $R$, substitute $2S=t^2+d^2+1$ into both sides of
\[
 t(S+d)+2ab(1-S)
 =ab(1-d^2)+t^2(1+d-ab)+\frac t2(t-1-d)^2
\]
and expand.  The two polynomials agree term by term.  The sign argument in Section~\ref{sec:concavity} then completes Sylvester's criterion without a numerical bound.

\section*{Statements and Declarations}

\paragraph{Competing interests.}
The author declares that he has no competing financial or non-financial interests that are directly or indirectly related to this work.

\paragraph{Funding.}
No specific funding was received for this work.

\paragraph{Data and code availability.}
No empirical dataset is used.  The submission includes the symbolic check scripts and numerical regression code referred to in Section~\ref{sec:numerics}.  None of these programs is a logical premise of the analytic proof.

\paragraph{Author contributions.}
Dawid Trela is the sole author of the work.

\paragraph{Ethics approval and consent.}
Not applicable.

\paragraph{AI-assisted tools.}
During preparation of the work, the author used OpenAI ChatGPT and Codex to assist with exploratory algebraic checks, code preparation, and manuscript editing.  The author independently checked the mathematical arguments, references, computations, and final text and takes full responsibility for the article.

\bibliographystyle{spmpsci}
\bibliography{references}

\end{document}